\documentclass[a4,12pt]{amsart}
\usepackage{amssymb,amstext,amsmath,amscd,amsthm,amsfonts,enumerate,latexsym}
\usepackage{color}
\usepackage{graphicx}
\usepackage[all]{xy}
\usepackage{stmaryrd} %mapsfrom

\usepackage{hyperref}
\usepackage{bm}

\theoremstyle{plain}
\newtheorem{thm}{Theorem}[section]

\newtheorem*{thm*}{Theorem}
\newtheorem*{cor*}{Corollary}

\newtheorem{prop}[thm]{Proposition}

\newtheorem{lem}[thm]{Lemma}
\newtheorem{cor}[thm]{Corollary}

\newtheorem{claim}{Claim}
\newtheorem*{claim*}{Claim}

\theoremstyle{definition}
\newtheorem{defn}[thm]{Definition}

\newtheorem{ex}[thm]{Example}

\newtheorem{fact}[thm]{Fact}
\newtheorem{conj}[thm]{Conjecture}

\newtheorem{ques}[thm]{Question}

\newtheorem{setup}[thm]{Setup}

\theoremstyle{remark}
\newtheorem{rem}[thm]{Remark}

\numberwithin{equation}{thm}

\def\Min{\operatorname{Min}}

\def\Ext{\operatorname{Ext}}

\def\Im{\operatorname{Im}}

\def\Hom{\operatorname{Hom}}

\def\m{\mathfrak m}

\def\p{\mathfrak p}

\newcommand{\rma}{\mathrm{a}}

\newcommand{\rme}{\mathrm{e}}

\newcommand{\rmo}{\mathrm{o}}

\newcommand{\rmr}{\mathrm{r}}

\newcommand{\rmH}{\mathrm{H}}
\newcommand{\rmI}{\mathrm{I}}

\newcommand{\rmK}{\mathrm{K}}

\newcommand{\rmQ}{\mathrm{Q}}

\newcommand{\calR}{\mathcal{R}}

\newcommand{\fka}{\mathfrak{a}}

\newcommand{\fkm}{\mathfrak{m}}

\newcommand{\fkM}{\mathfrak{M}}
\newcommand{\fkN}{\mathfrak{N}}

\newcommand{\mapright}[1]{%
\smash{\mathop{%
\hbox to 1cm{\rightarrowfill}}\limits^{#1}}}

\newcommand{\mapleft}[1]{%
\smash{\mathop{%
\hbox to 1cm{\leftarrowfill}}\limits_{#1}}}

\def\depth{\operatorname{depth}}

\def\height{\mathrm{ht}}
\def\grade{\mathrm{grade}}
\def\Spec{\operatorname{Spec}}
\def\tr{\operatorname{tr}}

\def\red{\operatorname{red}}

\def\adj{\operatorname{adj}}

\title[Nearly Gorenstein blow-up algebras]{Nearly Gorenstein blow-up algebras \\ over two-dimensional regular local rings}

\author[Naoki Endo]{Naoki Endo}
\address{School of Political Science and Economics, Meiji University, 1-9-1 Eifuku, Suginami-ku, Tokyo 168-8555, Japan}
\email{endo@meiji.ac.jp}
\urladdr{https://www.isc.meiji.ac.jp/~endo/}

\author[Ken-ichi Yoshida]{Ken-ichi Yoshida}
\address{Department of Mathematics, College of Humanities and Sciences, Nihon University, 3-25-40 Sakurajosui, Setagaya-Ku, Tokyo 156-8550, Japan}
\email{yoshida.kennichi@nihon-u.ac.jp}

\thanks{The first author was partially supported by JSPS Grant-in-Aid for Scientific Research (C) 23K03058 and 26K06773. The second author was partially supported by JSPS Grant-in-Aid for Scientific Research (C) 24K06678.}
\begin{document}

\maketitle

\setlength{\baselineskip} {15.3pt}

\begin{abstract}
In this paper, we slightly extend the framework of nearly Gorenstein graded rings introduced by Herzog-Hibi-Stamate, and investigate the nearly Gorenstein property of the blow-up algebras associated with an ideal, namely the Rees algebra, the extended Rees algebra, and the associated graded ring. We establish characterizations of the nearly Gorenstein property for these algebras and clarify the relations among them. 
In particular, for two-dimensional regular local rings, we obtain a complete classification of the integrally closed ideals for which these blow-up algebras are nearly Gorenstein.
\end{abstract}

%{\footnotesize \tableofcontents}

%%%%%%%%%%%%%%%%%%%%%%%%%%%%%%%%%%%%%%%%%%%%%%%%%%%%%%%%%%%%%%%%%%%%%%%%%%%%%%%%%%%%%%%%%%%%%%%%%%%%%%%%%%%%%%%%%%%%%%%%%%%%%%%%%%%%

%%%%%%%%%%%%%%%%%%%%%%%%%%%%%%%%%%%%%%%%%%%%%%%%%%%%%%%%%%%%%%%%%%%%%%%%%%%%%%%%%%%%%%%%%%%%%%%%%%%%%%%%%%%%%%%%%%%%%%%%%%%%%%%%%%%%%%%%%%%%%%%%%%%%%%

\section{Introduction}

In this paper, we investigate the nearly Gorenstein property of blow-up algebras. 
For an ideal $I$ in a Noetherian ring $A$, among the blow-up algebras associated with $I$, the Rees algebra $\calR(I)=\bigoplus_{n\geq 0} I^n$ is one of the fundamental objects lying at the interface of commutative algebra and algebraic geometry. Indeed, the algebra $\calR(I)$ is the homogeneous coordinate ring of the graph of a rational map whose total space is the blow-up of $\Spec A$ along the subscheme defined by $I$.
Since the pioneering work of \cite{GS2}, followed by substantial developments in, among others, \cite{Ikeda, GN}, the Gorenstein property of Rees algebras has been studied extensively. 
A fundamental result in this direction, known as {\it Goto-Shimoda's theorem}, asserts the following: if $(A,\m)$ is a Cohen-Macaulay local ring with $d=\dim A \ge 2$, and if $I$ is an $\m$-primary ideal of $A$, then $\calR(I)$ is Gorenstein if and only if the associated graded ring $G(I)=\bigoplus_{n\ge0} I^n/I^{n+1}$
is Gorenstein and has $a$-invariant $-2$ (\cite[Theorem (1.2)]{GS2}, \cite[Corollary 3.7]{Ikeda}, \cite[Part II, Corollary (1.4)]{GN}). 
These works have shown that being Gorenstein imposes severe restrictions on both the base ring and the ideal, so that this phenomenon occurs only under rather exceptional circumstances. Accordingly, it is natural to investigate blow-up algebras in a more flexible framework. This broader viewpoint is also essential for a deeper understanding of Gorenstein rings.

On the other hand, although the notion of trace ideals is classical and goes back at least to \cite{AG}, it has regained considerable attention in commutative algebra through the work of Lindo \cite{L,L2} and Lindo-Pande \cite{LP}. In particular, Herzog, Hibi, and Stamate introduced in \cite{HHS} the notion of nearly Gorenstein rings, observing that for a Cohen-Macaulay local ring $(A, \m)$, the trace ideal $\tr_A(\rmK_A)$ of the canonical module $\rmK_A$ detects the non-Gorenstein locus (\cite[Lemma 2.1]{HHS}). 
Thus, nearly Gorenstein rings provide a framework for studying Cohen-Macaulay rings that are close to Gorenstein rings via the trace of the canonical module, and their relationship with the almost Gorenstein property (\cite{BF, GMP, GTT}) has also become an important topic.

Along these lines, the authors, together with Goto and Matsuoka, have investigated the almost Gorenstein property of blow-up algebras, particularly Rees algebras (\cite{GMTY1, GMTY2, GMTY3, GMTY4}). 
In particular, it was shown in \cite{GMTY2} that for every $\m$-primary integrally closed ideal $I$ in a two-dimensional regular local ring $(A, \m)$ with infinite residue field, the Rees algebra $\calR(I)$ is an almost Gorenstein graded ring. 
Classically, by Zariski's theorem (\cite[Part II, Section 12]{Z}, \cite[Appendix 5, Theorem 2']{ZS}, \cite[Theorem 3.7]{Huneke}), the ring $\calR(I)$ is normal. 
In addition, it follows from \cite[Theorem 3.2]{Huneke-Sally} (see also \cite[Theorem]{Verma}) that $\calR(I)$ has minimal multiplicity at the graded maximal ideal $\fkM = \m \calR(I) + \calR(I)_+$. 
Hence, in view of \cite[Theorem 6.6]{HHS}, which asserts that a nearly Gorenstein ring admitting minimal multiplicity is almost Gorenstein, it is natural to ask for a precise characterization of when the Rees algebra $\calR(I)$ is nearly Gorenstein. This question is of particular interest, since it not only deepens our understanding of the almost Gorenstein property, but also clarifies the role played by the canonical trace in the context of blow-up algebras. 

The aim of this paper is to explore the nearly Gorenstein property of blow-up algebras. More specifically, we wish to understand whether Goto-Shimoda's theorem has an analogue in the nearly Gorenstein setting. This leads to the following question.

\begin{ques}\label{ques}
Let $(A, \m)$ be a Cohen-Macaulay local ring with $d = \dim A \ge 2$, and let $I$ be an $\m$-primary ideal of $A$. Assume that the Rees algebra $\calR(I)$ is nearly Gorenstein. 
\begin{itemize}
\item[$(1)$] Is the associated graded ring $G(I)$ nearly Gorenstein?
\item[$(2)$] Is it possible to determine the $a$-invariant of $G(I)$?
\end{itemize}
\end{ques}

In this paper we provide a complete answer to Question \ref{ques} $(2)$. Moreover, with a view toward applications to two-dimensional regular local rings, we concentrate on the case where the base ring is Gorenstein and the ideal has reduction number one; in this setting, we also answer Question \ref{ques} $(1)$ in the affirmative.

Let us now present a more detailed account of our main results and explain how this paper is organized. In Section 2, after recalling the definition and basic properties of trace ideals, we introduce the notion of the nearly Gorenstein property for $H$-local graded rings, that is, graded rings admitting a unique graded maximal ideal. This slightly extends the definition given in \cite[Definition 2.2]{HHS} for positively graded Cohen-Macaulay algebras over a field, and enables us to study the nearly Gorenstein property of blow-up algebras. We further show in Section 2 that the nearly Gorenstein property of the Rees algebra determines the $a$-invariant of the associated graded ring (Theorem \ref{2.10}). 
In Section 3, we establish a characterization of the nearly Gorenstein property of Rees algebras when the base ring is Gorenstein of dimension at least two and the ideal has reduction number one (Propositions \ref{2.13} and \ref{3.5}). In particular, in the case of a two-dimensional regular local ring, we completely determine the ideals whose Rees algebras are nearly Gorenstein.

\begin{thm}[Theorem \ref{3.3}]
Let $(A, \m)$ be a two-dimensional regular local ring with infinite residue class field. Let $I$ be an $\m$-primary ideal of $A$ and assume that $I$ is not a parameter ideal. Then the following conditions are equivalent. 
\begin{itemize}
\item[$(1)$]  The Rees algebra $\calR(I)$ is nearly Gorenstein.  
\item[$(2)$] There exists a regular system of parameters $x, y$ of $A$ and an integer $n \ge 1$ such that $I=(x^2, xy, y^{n+1})$. 
\end{itemize}
When this is the case, $\calR(I)$ is an almost Gorenstein graded normal domain with minimal multiplicity. 
\end{thm}

Moreover, Section 3 shows that, in the non-regular Gorenstein case, the nearly Gorenstein property of the Rees algebra forces the base ring to be a hypersurface of multiplicity 2 (Theorem \ref{hyp}).
In Section 4, we focus on extended Rees algebras $\calR'(I)=\bigoplus_{n\in \Bbb Z} I^n$. We give a characterization of their nearly Gorenstein property, and show that this is in fact equivalent to the nearly Gorenstein property of the associated graded ring (Proposition \ref{4.3}). As a consequence, we prove that the nearly Gorenstein property of the Rees algebra implies that of the extended Rees algebra (Corollary \ref{4.5}).
Furthermore, in the case of a two-dimensional regular local ring and the ideals are integrally closed, we obtain a complete classification of the ideals whose extended Rees algebras are nearly Gorenstein.

\begin{thm}[Theorem \ref{4.12}]
Let $(A, \m)$ be a two-dimensional regular local ring with infinite residue class field. Let $I$ be an integrally closed $\m$-primary ideal of $A$ and assume that $I$ is not a parameter ideal. Then the following conditions are equivalent. 
\begin{itemize}
\item[$(1)$] The extended Rees algebra $\calR'(I)$ is nearly Gorenstein.  
\item[$(2)$] The associated graded ring $G(I)$ is nearly Gorenstein.  
\item[$(3)$] Either $I=\m^{\ell}$ for some $\ell \ge 2$, or else there exists a regular system of parameters $x, y$ of A and an integer $n \ge 1$ such that $I=(x^2, xy, y^{n+1})$. 
\end{itemize}
\end{thm}

Throughout this paper, unless otherwise specified, we use the following notation and terminology. For an arbitrary commutative ring $A$ and an $A$-module $M$, let $\ell_A(M)$ denote the length of $M$. We denote by $\rmQ(A)$ the total ring of fractions of $A$. For $A$-submodules $X$ and $Y$ of $\rmQ(A)$, we set $X:_{\rmQ(A)}Y = \{a \in \rmQ(A) \mid aY \subseteq X\}$. For ideals $I, J$ of $A$, we set $I:_AJ =\{a \in A \mid aJ \subseteq I\}$; hence $I:_AJ = (I:_{\rmQ(A)}J) \cap A$. A fractional ideal of $A$ is a finitely generated $A$-submodule $X$ of $\rmQ(A)$ such that $\rmQ(A)\cdot X = \rmQ(A)$. 
For an ideal $I$ of $A$, we denote by $\overline{I}$ the integral closure of $I$. 
For an integer $t > 0$, let $\rmI_t(\Bbb M)$ be the ideal of $A$ generated by all the $t \times t$ minors of a matrix $\Bbb M$ with entries in $A$. 
When $(A, \m)$ is a Cohen-Macaulay local ring with $d = \dim A$ and $M$ is an $A$-module, let $\mu_A(M)$ denote the minimal number of generators of $M$.
We denote by $\rme(A)$ the multiplicity of $A$, and set $\rmr(A) = \ell_A(\Ext^d_A(A/\m, A))$.
For a graded module $M$ over a graded ring $R$ and an integer $\ell$, let $M(\ell)$ denote the graded $R$-module whose underlying $R$-module is $M$ and whose grading is given by $[M(\ell)]_m = M_{\ell + m}$ for all $m \in \Bbb Z$, where $[-]_m$ denotes the $m$-th homogeneous component.

%%%%%%%%%%%%%%%%%%%%%%%%%%%%%%%%%%%%%%%%%%%%%%%%%%%%%%%%%%%%%%%%%%%%%%%%%%%%%%%%%%%%%%%%%%%%%%%%%%%%%%%%%%%%%%%%%%%%%%%%%%%%%%%%%%%%%%%%%%%%%%%%%%%%%%

\section{Preliminaries}

We give the definitions and basic properties which we need throughout this paper. 

\begin{defn}
Let $R$ be an arbitrary commutative ring and $M$ be an $R$-module. 
Consider the homomorphism of $R$-modules
$$
\tau_M: \Hom_R(M, R) \otimes_RM \to R
$$ 
defined by $\tau_M(f \otimes m)=f(m)$ for each $f \in \Hom_R(M,R)$ and $m \in M$. Set $\tr_R(M)= \Im \tau_M$, which forms an ideal of $R$, and call it the {\it trace ideal of $M$}.  The reader may consult with \cite{DL, DKT, GIK, HHS, L, LP} for more details regarding trace ideals.
\end{defn}

\begin{lem}\label{2.2}
Let $R=\bigoplus_{n \in \Bbb Z}R_n$ be a Noetherian $\Bbb Z$-graded ring and $M$ a finitely generated graded $R$-module. For each $\ell \in \Bbb Z$, the equality $\tr_R(M) = \tr_R(M(\ell))$ holds. 
\end{lem}

\begin{proof}
This follows from the natural isomorphisms 
$$
\Hom_R(M(\ell), R) \otimes_R M(\ell) \cong \Hom_R(M, R)(-\ell)\otimes_R M(\ell)\cong \Hom_R(M, R) \otimes_R M,
$$
under which the evaluation map $\tau_{M(\ell)}$ corresponds to $\tau_M$. 
Hence $\tr_R(M)=\tr_R(M(\ell))$.
\end{proof}

In \cite[Definition 2.2]{HHS}, the nearly Gorenstein property is defined for Cohen-Macaulay local rings that possess a canonical module, as well as for positively graded Cohen-Macaulay algebras over a field.  The present work requires a slightly broader viewpoint for graded rings with a unique graded maximal ideal. 
We therefore extend the notion accordingly and examine its consequences for graded rings.

\begin{defn}\label{2.3}
Let $R=\bigoplus_{n \in \Bbb Z}R_n$ be a Cohen-Macaulay $\Bbb Z$-graded ring. Assume that $R$ is $H$-local, i.e., it has a unique graded maximal ideal $\fkM$ of $R$, and that $R$ admits the graded canonical module $\rmK_R$. We say that $R$ is {\it nearly Gorenstein} if $\fkM \subseteq \tr_R(\rmK_R)$.
\end{defn}

Note that an $H$-local graded ring is also referred to as ${}^\ast$local. The $R$-module $\rmK_R$ is the canonical module of $R$ in the usual (non-local) sense, i.e., for every prime ideal $\p$ of $R$, the localization $(\rmK_R)_\p$ is a canonical module of $R_\p$. Moreover, $\rmK_R$ is uniquely determined up to  homogeneous isomorphism (\cite[Proposition 4.2]{BB}, \cite[Proposition 3.6.9]{BH}). In particular, one has $(\rmK_R)_{\fkM} \cong \rmK_{(R_\fkM)}$, and $R$ is nearly Gorenstein if and only if so is the local ring $R_{\fkM}$.

\begin{ex}\label{2.4}
Let $R=k[X_1,\ldots,X_d]~(d\ge 2)$ be the polynomial ring over a field $k$, and set $\m=(X_1,\ldots,X_d)$. The Rees algebra $\calR(\m)=\bigoplus_{n \ge 0}\m^n$  admits the presentation
$$
\calR(\m) \cong \small k[X_1, X_2, \ldots, X_d, Y_1, Y_2, \ldots, Y_d]/{\rm I_2}
\begin{pmatrix}
X_1 & X_2 & \cdots & X_d \\
Y_1 & Y_2 & \cdots & Y_d
\end{pmatrix}
$$
where $\rmI_2(\Bbb M)$ denotes the ideal generated by the $2\times 2$-minors of a matrix $\Bbb M$. 
Identifying $\calR=\calR(\m)$ with the above determinantal ring, we may apply \cite[Theorem 1.1]{FHST} to compute the trace of its graded canonical module $\rmK_{\calR}$. More precisely, one has the equality
$$
\tr_{\calR}(\rmK_{\calR}) = (X_1, X_2, \ldots, X_d, Y_1, Y_2, \ldots, Y_d)^{d-2}\calR.
$$
Hence, $\calR(\m)$ is nearly Gorenstein if and only if $d\le 3$.
\end{ex}

In the rest of this section, let $(A, \m)$ be a Cohen-Macaulay local ring with $d = \dim A >0$ and $I~(\ne A)$ an ideal of $A$ with $\height_AI>0$. Denote by
$$
\calR(I) = A[It] \subseteq A[t],  \ \calR'(I) = A[It, t^{-1}] \subseteq A[t, t^{-1}], \  \text{and} \ G(I) = \calR'(I)/t^{-1}\calR'(I) 
$$
the {\it Rees algebra} of $I$, the {\it extended Rees algebra} of $I$, and the {\it associated graded ring} of $I$, respectively, where $t$ is an indeterminate over $A$. For background and basic results on these algebras, the reader is referred to \cite[Section4.5]{BH}, \cite{HIO}, \cite[Chapter 5]{SH}, \cite{Vasconcelos}, and \cite{Vasconcelos2}.

Let $\fkM = \m\calR(I) + \calR(I)_+$ denote the graded maximal ideal of $\calR(I)$ and 
$$
\rma(G(I)) = \sup\{n \in \Bbb Z \mid [\rmH^d_\fkM(G(I))]_n \ne (0) \}
$$
be the $a$-invariant of $G(I)$ (\cite[Definition (3.1.4)]{GW1}), where $\{[\rmH^d_\fkM(G(I))]_n\}_{n \in \Bbb Z}$ denotes the homogeneous components of the $d$-th graded local cohomology module $\rmH^d_\fkM(G(I))$ of $G(I)$ with respect to $\fkM$.

\begin{fact}[{\cite[Theorem 7.1]{Trung-Ikeda}}]\label{2.5} 
The following conditions are equivalent. 
\begin{itemize}
\item[$(1)$] $\calR(I)$ is a Cohen-Macaulay ring.
\item[$(2)$] $G(I)$ is a Cohen-Macaulay ring and $\rma(G(I)) < 0$. 
\end{itemize}
\end{fact}

\begin{rem}\label{2.6}
We denote by $S$ (resp. $T$) the set of all the homogeneous non-zerodivisors on $\calR(I)$ (resp. $\calR'(I)$). Since $\grade_AI > 0$, it follows that
$$
S^{-1}\calR(I) = K[t, t^{-1}] = T^{-1}\calR'(I)
$$
where $K=\rmQ(A)$ denotes the total ring of fractions of $A$. 
\end{rem}

Introduced by Northcott and Rees (\cite{NR}), for ideals $J$ and $Q$ of $A$ with $Q \subseteq J$, we say that $Q$ is a {\it reduction} of $J$ if $J^{r+1} = QJ^r$ for some $r \ge 0$; the least such integer $r$, denoted by $\red_Q(J)$, is called the {\it reduction number} of $J$ with respect to $Q$.   
The following lemma is well-known to experts; for the sake of completeness, we include a brief proof.

\begin{lem}[{cf. \cite[pages 202--203]{GS2}}]\label{2.7}
%Let $(A, \m)$ be a Cohen-Macaulay local ring with $d = \dim A >0$. Let $I$ be an $\m$-primary ideal of $A$ 
Let $I$ be an $\m$-primary ideal of $A$. Assume that $I$ contains a parameter ideal $Q$ of $A$ as a reduction and $G(I)$ is a Cohen-Macaulay ring. Then $\rma(G(I)) = \red_Q(I)-d$. 
\end{lem}

\begin{proof}
Let $Q=(a_1, a_2, \ldots, a_d)$. Since $Q$ is a reduction of $I$, the sequence $a_1t, a_2t, \ldots, a_dt$ constitutes a homogeneous system of parameters of $G(I)$. As $G(I)$ is Cohen-Macaulay, it forms a regular sequence on $G(I)$. Hence, by \cite[Corollary 2.7]{VV} we have $Q \cap I^n = QI^{n-1}$ for all $n \in \Bbb Z$. It follows that $G(I)/(a_1t, a_2t, \ldots, a_dt) \cong G(I/Q)$ and $\rma(G(I/Q)) = \red_Q(I)$. 
Finally, \cite[Remark (3.1.6)]{GW1} yields 
$$
\rma(G(I)) = \rma(G(I)/(a_1t, a_2t, \ldots, a_dt)) - d = \red_Q(I)-d
$$
as desired.
\end{proof}

\begin{rem}\label{2.8}
When $d=2$ and $I$ is $\m$-primary, Lemma \ref{2.7} together with \cite[Corollary 2.7]{VV} shows that $\calR(I)$ is Cohen-Macaulay if and only if $I^2=QI$. Moreover, a necessary and sufficient condition for $\calR(I)$ to be a Gorenstein ring is that $G(I)$ is Gorenstein and $\rma(G(I)) =-2$ (\cite[Corollary 3.7]{Ikeda}). 
Assuming that $A$ is Gorenstein, this pair of conditions is equivalent to $\red_Q(I)=0$, that is, $I=Q$.
\end{rem}

Recall that the ring $A$ is {\it generically Gorenstein} if $A_{\p}$ is Gorenstein for every $\p \in \Min A$.

\begin{lem}\label{2.9}
Assume that $A$ admits the canonical module $\rmK_A$. If $\calR(I)$ is nearly Gorenstein, then $A$ is generically Gorenstein. 
\end{lem}

\begin{proof}
Set $\calR=\calR(I)$. Note that $\calR$ admits a graded canonical module $\rmK_{\calR}$. 
Let $\p \in \Min A$ and set $P=\p A[t] \cap \calR$. Then $P \cap A = \p$. Since $\height_A I>0$, we have $I_\p=A_\p$ and hence $\calR_\p \cong \calR_{A_\p}(I_\p) \cong A_\p[t]$. Therefore
$$
\calR_P \cong (\calR_\p)_{P \calR_\p} \cong (A_\p[t])_{\p A_\p[t]}.
$$
In particular, $\dim \calR_P=\dim A_\p=0$, so $P\in\Min \calR$. 
Thus $P\ne \fkM$. Since $\calR$ is nearly Gorenstein, \cite[Proposition 2.3 (a)]{HHS} implies that $\calR_P$ is Gorenstein. Via the above isomorphisms, $(A_\p[t])_{\p A_\p[t]}$ is Gorenstein, and hence so is $A_\p$. Consequently, the local ring $A$ is generically Gorenstein.
\end{proof}

\begin{thm}\label{2.10}
Let $(A, \m)$ be a Cohen-Macaulay local ring with $d = \dim A >0$ admitting the canonical module $\rmK_A$. Let $I~(\ne A)$ be an ideal of $A$ with $\height_AI>0$ and assume that $\calR(I)$ is nearly Gorenstein.  
Then the following assertions hold true. 
\begin{itemize}
\item[$(1)$] One has $\rma(G(I)) = -1, -2, -3$. 
\item[$(2)$] If $\rma(G(I)) \le -2$, then $A$ is nearly Gorenstein. 
\item[$(3)$] If $\rma(G(I)) = -3$, then $I=\m$. 
\end{itemize}
\end{thm}

\begin{proof}
Set $\calR=\calR(I)$, $G=G(I)$, and $a = \rma(G(I))$. By Fact \ref{2.5}, the ring $G$ is Cohen-Macaulay and $a < 0$. 
Since $A$ is generically Gorenstein, we realize the canonical module as an ideal of $A$ with $\grade_A\rmK_A=1$ (\cite[Satz~6.21]{HK}, \cite[Remark~3.2]{GTT}, and \cite[Proposition~3.3.18]{BH}).
By \cite[Theorem 2.1 and Proposition 2.5]{GI}, there exists a unique family $\omega=\{\omega_i\}_{i \in \Bbb Z}$ of $A$-submodules of $\rmK_A$ satisfying the following conditions: 
\begin{itemize}
\item[$(a)$] $\omega_{i+1} \subseteq \omega_i$ for all $i \in \Bbb Z$.
\item[$(b)$] $\omega_i=\rmK_A$ for all $i\le -a-1$, and $\omega_{-a}\subsetneq \rmK_A$.
\item[$(c)$] $I^i \omega_j \subseteq \omega_{i+j}$ for all $i, j \in \Bbb Z$.
\item[$(d)$] $\rmK_{\calR}\cong \sum_{i\ge 1}\omega_i t^i$, where $\rmK_{\calR}$ denotes the graded canonical module of $\calR$.
\end{itemize}
Identifying $\rmK_{\calR}$ with $\sum_{i\ge 1}\omega_i t^i$, we regard $\rmK_{\calR}$ as a graded $\calR$-submodule of $S^{-1}\calR = K[t,t^{-1}]$, where $S$ is the set of homogeneous non-zerodivisors on $\calR$ and $K=\rmQ(A)$ stands for the total ring of fractions of $A$. 
In particular, $\rmK_{\calR}$ is a fractional ideal of $\calR$.

Let $F=\rmQ(\calR)$ be the total ring of fractions of $\calR$, and set $L=\calR:_F \rmK_{\calR}$. 
We claim that
\[
L=\calR:_{S^{-1}\calR}\rmK_{\calR}.
\]
Indeed, since $a<0$, condition $(b)$ gives $\omega_0=\rmK_A$. Hence, by $(c)$ we get $I\rmK_A = I\omega_0 \subseteq \omega_1$. 
Because $\height_AI>0$, the ideal $I\rmK_A$ contains a non-zerodivisor $c$ on $A$. Then $ct\in \omega_1 t\subseteq \rmK_{\calR}$ is a homogeneous non-zerodivisor on $\calR$. Moreover, since $ct\in It$, we have $ct\in S$. Pick $\alpha\in L$. Then $\alpha\,\rmK_{\calR}\subseteq \calR$, and hence in particular $\alpha\cdot(ct)\in\calR$. Therefore
$$
\alpha=\frac{\alpha\cdot(ct)}{ct}\in S^{-1}\calR.
$$
This shows that $L\subseteq S^{-1}\calR$, proving the claim. Consequently, $L$ is a graded $\calR$-submodule of $S^{-1}\calR$, and we may write $L=\sum_{n\in\Bbb Z} L_n t^n$.

$(1)$ Let $n\le -2$ be an integer. Since $L$ is graded, we have 
$L_n t^n\cdot (\omega_1 t) \subseteq \calR_{n+1}=(0)$. 
As $\omega_1$ contains a non-zerodivisor on $A$, it follows that $L_n=(0)$ for all $n\le -2$. Since $\rmK_{\calR}$ is a fractional ideal of $\calR$, we have
$\tr_{\calR}(\rmK_{\calR})=\rmK_{\calR}\cdot L$. 
In particular, using $L_n=(0)$ for $n\le -2$, we obtain
$[\tr_{\calR}(\rmK_{\calR})]_0=\omega_1\cdot L_{-1}$. 
Assume, to the contrary, that $a\le -4$. Then condition $(b)$ gives
$\omega_1=\omega_2=\omega_3=\rmK_A$. 
Since $\calR$ is nearly Gorenstein, we have $\m=\fkM_0\subseteq [\tr_{\calR}(\rmK_{\calR})]_0$, and hence
$$
\m \subseteq \omega_1\cdot L_{-1}=\rmK_A\cdot L_{-1}=\omega_3\cdot L_{-1}.
$$
Multiplying by $t^2$, we see that $\m t^2 \subseteq (\omega_3 t^3)\cdot (L_{-1}t^{-1}) \subseteq \calR_2 = I^2 t^2$, 
and therefore $\m\subseteq I^2$. This contradicts $d >0$ (indeed, $\m\subseteq I^2$ forces $I=\m$ and hence $\m=\m^2$, so $A$ would be Artinian). 
Consequently, $a=\rma(G)$ must be one of $-1$, $-2$, or $-3$.

$(2)$ Assume that $a\le -2$. Then condition $(b)$ yields $\omega_1=\rmK_A$. 
Taking the degree $0$ part of $L\cdot \rmK_{\calR}\subseteq \calR$, we get
$
L_{-1}t^{-1}\cdot (\omega_1 t)\subseteq \calR_0=A,
$
and hence $L_{-1}\subseteq A:_K \rmK_A$. 
Therefore
$$
\m =\fkM_0 \subseteq [\tr_{\calR}(\rmK_{\calR})]_0 = \omega_1 t\cdot L_{-1}t^{-1}
= \rmK_A\cdot L_{-1}
\subseteq \rmK_A\cdot (A:_K \rmK_A)
=\tr_A(\rmK_A).
$$
Thus $A$ is nearly Gorenstein.

$(3)$ Assume that $a=-3$. By condition $(b)$, we see that $\omega_1=\omega_2=\rmK_A$. It follows that
\[
\m=\fkM_0 \subseteq [\tr_{\calR}(\rmK_{\calR})]_0
=\omega_1\cdot L_{-1}
=\rmK_A\cdot L_{-1}
=\omega_2\cdot L_{-1}.
\]
Multiplying by $t$, we obtain $\m t\subseteq (\omega_2 t^2) \cdot (L_{-1}t^{-1})\subseteq \calR_1 = I t$. Hence $\m\subseteq I$ and we conclude that $I=\fkm$.  This completes the proof.
\end{proof}

\begin{ex}\label{2.11}
Let $R=k[X_1, \ldots, X_d]~(d \ge 3)$ be the polynomial ring over a field $k$, endowed with the standard grading in which $\deg X_i=1$ and $R_0=k$.
Fix a homogeneous polynomial $f\in R$ of degree $d-3$ and set $B=R/(f)$. 
Let $M=(X_1, \ldots, X_d)B$ be the graded maximal ideal of $B$, and consider the local ring $A = B_M$. Write $\m=MA$ and set $G=G(\m)$. Then $A\cong G_{\fkN}$, where $\fkN$ denotes the graded maximal ideal of $G$.
Set $a=a(G)$. By \cite[Remark~(3.1.6)]{GW1}, we have $a=-d+\deg f=-3$. In particular, the Rees algebra $\calR=\calR(\m)$ is Cohen-Macaulay but not Gorenstein. 
Following \cite[Theorem 2.1 and Proposition 2.5]{GI}, let $\omega=\{\omega_i\}_{i\in\Bbb Z}$ denote the uniquely determined family of ideals of $A$ satisfying:
\begin{itemize}
\item[$(a)$] $\omega_{i+1}\subseteq \omega_i$ for all $i\in\Bbb Z$.
\item[$(b)$] $\omega_i=A$ for all $i\le -a-1$, and $\omega_{-a}\subsetneq A$.
\item[$(c)$] $\m^i\omega_j\subseteq \omega_{i+j}$ for all $i,j\in\Bbb Z$.
\item[$(d)$] $\rmK_{\calR}\cong \sum_{i\ge 1}\omega_i t^i$.
\item[$(e)$] $\rmK_G\cong \bigoplus_{i\ge 1}\omega_{i-1}/\omega_i$.
\end{itemize}
Here, $\rmK_{\calR}$ and $\rmK_G$ denote the graded canonical modules of $\calR$ and $G$, respectively. 
Since $G$ is Gorenstein, a comparison of graded components in $\rmK_G\cong G(a)$ yields $\omega_i=\m^{i+a+1}$ for all $i\ge -a$.
Let $F=\rmQ(\calR)$ be the total ring of fractions of $\calR$, and set $L=\calR:_F \rmK_{\calR}$. 
A direct computation shows that $L_n=\m^{n-a-1}$ for all $n\ge -1$, while $L_n=(0)$ for all $n\le -2$. Consequently
$$
[\tr_{\calR}(\rmK_{\calR})]_0 = \omega_1 \cdot L_{-1}=\m  \ \ \ \text{and} \ \ \ [\tr_{\calR}(\rmK_{\calR})]_1 = \omega_2t^2 \cdot L_{-1}t^{-1} + \omega_1t\cdot L_0=\m t. 
$$
Therefore $\calR$ is nearly Gorenstein but not Gorenstein.
\end{ex}

%%%%%%%%%%%%%%%%%%%%%%%%%%%%%%%%%%%%%%%%%%%%%%%%%%%%%%%%%%%%%%%%%%%%%%%%%%%%%%%%%%%%%%%%%%%%%%%%%%%%%%%%%%%%%%%%%%%%%%%%%%%%%%%%%%

\section{Nearly Gorenstein Rees algebras}

In this section, we explore the nearly Gorenstein property of Rees algebras. Let us fix the notation and standing assumptions.

\begin{setup}
Let $(A, \m)$ be a Gorenstein local ring with $d = \dim A \ge 2$ and $K=\rmQ(A)$ its total ring of fractions. Let $I$ be an $\m$-primary ideal of $A$ and assume that $I$ contains a parameter ideal $Q$ of $A$ such that $I^2=QI$. Let 
$$
\calR=\calR(I) =A[It] \subseteq A[t]
$$
be the Rees algebra of $I$, where $A[t]$ denotes the polynomial ring over $A$. We denote by $\fkM = (\m, It)\calR$ the graded maximal ideal of $\calR$. 
Set $J=Q:_AI$. In particular, $I \subseteq J$. 
\end{setup}

By \cite[Theorem 2.7 (a)]{U}, the graded canonical module of $\calR$ is given by $$
\rmK_{\calR} \cong ((1,t)^{d-3}t, Jt^{d-1})\calR = \bigoplus_{i=1}^{d-2}At^i \oplus \bigoplus_{i \ge d-1}I^{i-d+1}Jt^i
$$ as graded $\calR$-modules. Hence, $\calR$ is Gorenstein if and only if $I=Q$ when $d=2$, whereas for $d \ge 3$, it is Gorenstein if and only if $I=J$ and $d=3$ (\cite[Corollary 2.11 (a)]{U}). 
Since the situations in dimension two and in dimension at least three are somewhat different, we treat them separately. We begin with the case where $d \ge 3$. 

Set $X=((1,t)^{d-3}, Jt^{d-2})\calR$, which is a fractional ideal of $\calR$. Then $\rmK_{\calR}(1)\cong X$, so that
$$
\tr_{\calR}(\rmK_{\calR}) = X \cdot L
$$
where $L=\calR:_F X$ and $F=\rmQ(\calR)$. Since $1\in X$, we actually have $L=\calR:_\calR X$, so $L$ is a graded ideal of $\calR$.

\begin{lem}\label{2.12}
For each $n \ge 0$,  the equality $L_n = [I^{n+d-3} \cap (I^{n+d-2}:_AJ)]t^n$ holds in $\calR$. 
\end{lem}

\begin{proof}
Let $x\in \calR_n$ and write $x=ct^n$ with $c\in I^n$. Then $x \in L_n$ if and 
only if 
$x(1,t)^{d-3}\subseteq \calR$ and $x(Jt^{d-2})\subseteq \calR$. 
The first condition holds if and only if $(ct^n)t^{d-3}\in \calR_{n+d-3}$, 
that is, $c\in I^{n+d-3}$. The second condition is equivalent to $(ct^n)(Jt^{d-2})\subseteq \calR_{n+d-2}$, or equivalently $c\in I^{n+d-2}:_A J$.
Therefore $x\in L_n$ if and only if $c\in I^{n+d-3}\cap (I^{n+d-2}:_A J)$, and hence $L_n=\bigl[I^{n+d-3}\cap (I^{n+d-2}:_A J)\bigr]t^n$, as desired.
\end{proof}

\begin{prop}\label{2.13}
Let $(A, \m)$ be a Gorenstein local ring with $d = \dim A \ge 3$ and $I$ an $\m$-primary ideal of $A$. Assume that $I$ contains a parameter ideal $Q$ of $A$ such that $I^2=QI$, and set $J=Q:_A I$. Then the following assertions hold true. 
\begin{itemize}
\item[$(1)$] If the Rees algebra $\calR(I)$ is nearly Gorenstein, then $d \le 4$. 
\item[$(2)$] Suppose that $d=3$. Then $\calR(I)$ is nearly Gorenstein but not Gorenstein if and only if $\m = I:_AJ$ and $I=(I^2:_AJ) + \m J$.
\item[$(3)$] Suppose that $d=4$. Then $\calR(I)$ is not Gorenstein, and it 
 is nearly Gorenstein if and only if $I=\m$.
\end{itemize}
\end{prop}

\begin{proof}
$(1)$ Since $G(I)$ is Cohen-Macaulay by Fact \ref{2.5}, Lemma \ref{2.7} yields $\rma(G(I))=\red_Q(I)-d \le 1-d$. Applying Theorem \ref{2.10} (1), we conclude that $d\le 4$.

$(2)$ Suppose that $d=3$. From Lemma \ref{2.12} we have 
$$
[\tr_{\calR}(\rmK_{\calR})]_0 =I:_AJ \ \  \text{and} \ \ [\tr_{\calR}(\rmK_{\calR})]_1 =(I \cap (I^2 :_AJ))t + J\cdot (I:_AJ) t. 
$$
Since $\fkM=(\m, It)\calR$, it follows that $\calR$ is nearly Gorenstein but not Gorenstein if and only if $\m = I:_AJ$ and $I=(I \cap (I^2:_AJ)) + \m J$. 
Observe that $I^2:_AJ \subseteq  I$, because $A/Q$ is an Artinian Gorenstein ring and $I^2=QI$. Hence the equivalence follows.

$(3)$ Suppose that $d=4$. Arguing as in $(2)$, the ring $\calR$ is nearly Gorenstein if and only if $\m=I\cap (I^2:_AJ)$ and $It = (I^2 \cap (I^3:_AJ))t + (I \cap (I^2:_AJ))t$. The second condition is equivalent to $I=I\cap (I^2:_A J)$. Combined with $\m=I\cap (I^2:_A J)$, this yields $I=\m$, and vice versa. Hence $\calR$ is nearly Gorenstein but not Gorenstein if and only if $I=\m$. %This completes the proof. 
\end{proof}

We now focus in particular on the case where $d=2$. By Fact \ref{2.5}, we observe that $\calR$ is Cohen-Macaulay. Note that $\rmK_\calR(1) \cong J \calR$ (\cite[Proposition 2.1]{GMTY2}, \cite[Theorem 2.7 (a)]{U}). 
Let $L=\calR:_F J\calR$. Then $L=\calR:_{S^{-1}\calR} J\calR$; see Remark \ref{2.6}.  In particular, $L$ is a graded $\calR$-submodule of $S^{-1}\calR$.

\begin{lem}
For each $n \in \Bbb Z$,  the equality $L_n =  (I^n:_AJ)t^n$ holds in $A[t, t^{-1}]$. 
\end{lem}

\begin{proof}
Let $x \in K[t, t^{-1}]$ be a homogeneous element of degree $n$ and write $x = ct^n$ with $c \in K$. Then $x \in L_n$ if and only if $(ct^n)(JI^{\ell}t^{\ell}) \subseteq I^{n+\ell}t^{n+\ell}$ for all $\ell \ge 0$. The latter condition is equivalent to saying that $c  J \subseteq I^n$, that is, $c \in I^n:_K J = I^n:_A J$. Hence $L_n =  (I^n:_AJ)t^n$, as claimed. 
\end{proof}

\begin{rem}
Let $A$ be a Noetherian ring and $J$ an ideal of $A$ with $\grade_AJ \ge 2$. Then $A:_{\rmQ(A)}J=A$ (see e.g., \cite[Exercise 1.2.24]{BH}, \cite[Remark 1.2]{HHS}). %where $\rmQ(A)$ is the total ring of fractions of $A$. 
Hence, for any ideal $\fka$ in $A$, it follows that $\fka :_{\rmQ(A)} J = \fka:_AJ$.
\end{rem}

Hence we have the following. 

\begin{prop}\label{3.5}
Let $(A, \m)$ be a two-dimensional Gorenstein local ring  and $I$ an $\m$-primary ideal of $A$. Assume that $I$ contains a minimal reduction $Q$  such that $\overline{Q} \ne Q$. Then the following conditions are equivalent. 
\begin{itemize}
\item[$(1)$] The Rees algebra $\calR(I)$ is nearly Gorenstein but not Gorenstein.
\item[$(2)$]  $I=Q:_A\m$ and $I=\m (I:_A\m)$. 
\end{itemize}
\end{prop}

\begin{proof}
By Fact \ref{2.5}, Lemma \ref{2.7}, and \cite[Theorem 2.1]{CHV}, we may assume $I^2=QI$. Furthermore, we may (and do) assume that $I\ne Q$; see Remark \ref{2.8}. 
Since $\fkM = (\m, It)\calR$ and $\tr_{\calR}(\rmK_\calR) = J\calR \cdot L$, the 
Rees algebra $\calR$ is nearly Gorenstein if and only if 
$$
\m \subseteq J  (A:_AJ) \ \ \text{and} \ \ It \subseteq (JIt)(A:_AJ) + J(I:_AJ)t.
$$
Equivalently, this is the case precisely when $\m \subseteq J$ and $I \subseteq J(I:_AJ)$. Consequently, because $I\neq Q$, we see that $\calR$ is nearly Gorenstein if and only if
$$
\m=J \ \ \text{and} \ \ I=\m(I:_A\m).
$$
Moreover, since $A/Q$ is Gorenstein, the equality $\m=J$ is equivalent to $I=Q:_A\m$.
%Hence, the equivalent condition follows from Proposition \ref{3.4}. 
\end{proof}

%%%%%%%%%%%%%%%%%%%%%%%%%%%%%%%%%%%%%%%%%%%%%%%%%%%%%%%%%%%%%%%%%%%%%%%%%%%%%%%%%%%%%%%%

Recall that  a Cohen-Macaulay local ring $(A, \m)$ is said to have {\it minimal multiplicity} if $\rme(A) = \mu_A(\m) - \dim A + 1$, where $\rme(A)$ denotes the multiplicity of $A$. If the residue field $A/\m$ is infinite, then $A$ has minimal multiplicity if and only if $\m^2=Q\m$ for some parameter ideal $Q$ of $A$ (\cite[Theorem 1]{S1}). Whenever $\calR(I)$ is Cohen-Macaulay, we say that $\calR(I)$ has {\it minimal multiplicity} if the local ring $\calR(I)_{\fkM}$ has minimal multiplicity. 

For each ideal $I$ of $A$, let 
$$
\rmo(I) = \max\{n \in \Bbb Z \mid I \subseteq \m^n\}
$$
denote the {\it order} of $I$. 
Let $(A, \m)$ be a two-dimensional regular local ring and $I$ an $\m$-primary ideal of $A$. Following \cite[Appendix 5, page 368]{ZS}, we say that $I$ is {\it contracted} if there exists $x \in \m \setminus \m^2$ such that 
$$
I\cdot A\left[\frac{\m}{x}\right] \cap A = I
$$
where $\frac{\m}{x} = \{\frac{a}{x} \mid a \in \m\}$. When $A/\m$ is infinite,  $I$ is contracted if and only if $\mu_A(I) = \rmo(I) + 1$ (\cite[Theorem 2.1]{Huneke-Sally}). Moreover, by \cite[Appendix 5, Lemma 2]{ZS}, the ideal $I$ being contracted is equivalent to the existence of $x \in \m \setminus \m^2$ satisfying $I:_A \m = I:_A x$.

\begin{lem}\label{3.1}
Let $(A, \m)$ be a two-dimensional regular local ring with infinite residue class field. Let $I$ be an $\m$-primary ideal of $A$ and assume that $I$ is not a parameter ideal. 
 Then the following conditions are equivalent. 
\begin{itemize}
\item[$(1)$] $I$ is integrally closed and $\calR(I)$ is a nearly Gorenstein ring.  
\item[$(2)$] $\calR(I)$ is nearly Gorenstein with minimal multiplicity. 
\item[$(3)$] There exists a regular system of parameters $x, y$ of $A$ and an integer $n \ge 1$ such that $I=(x^2, xy, y^{n+1})$. 
\end{itemize}
When this is the case, $\calR(I)$ is an almost Gorenstein graded ring  with $\rmr(\calR(I))=2$. 
\end{lem}

\begin{proof}
$(1) \Rightarrow (2)$ This follows from \cite[Theorem 3.2]{Huneke-Sally}. %; see also \cite[Theorem, page 3]{Verma}.

$(2) \Rightarrow (3)$ Let $Q$ be a minimal reduction of $I$. By Lemma \ref{2.7}, the equality $I^2=QI$ holds. Since $\calR$ is nearly Gorenstein, we then have $I=Q:_A\m$ and $I=\m(I:_A\m)$. In particular, $\mu_A(I) = 3$ because $A/Q$ is Gorenstein.  Note that, by \cite[Theorem, page 3]{Verma} the ideal $I$ is contracted, so that $\mu_A(I) = \rmo(I) + 1$. Hence $\rmo(I) = 2$. Moreover, because $I=\m(I:_A\m)$, it follows that $\rmo(I:_A \m) = 1$. Therefore we can choose a regular system of parameters $x, y$ of $A$ and an integer $n \ge 1$ such that $I:_A \m=(x, y^n)$. Consequently, $I = \m (I:_A \m) =(x^2, xy, y^{n+1})$. 

$(3) \Rightarrow (1)$ We choose a regular system of parameters $x, y$ of $A$ and an integer $n \ge 1$ satisfying that $I=(x^2, xy, y^{n+1})$. Since $I = \m(x, y^n)$, we have $I=\m (I:_A\m)$. By setting $Q=(xy, x^2+y^{n+1})$, the direct computation shows $I^2=QI$ and $\m I \subseteq Q$. Hence $I=Q:_A\m$, so that $\calR$ is nearly Gorenstein. In addition, since the products of integrally closed ideals is integrally closed (\cite[Appendix 5, Theorem 2']{ZS}), the ideal $I$ is integrally closed. The last statement follows from \cite[Theorem 1.4]{GMTY3}. 
\end{proof}

\begin{rem}
By \cite[Theorem 6.6]{HHS}, the local ring $\calR(I)_{\fkM}$, being nearly Gorenstein with minimal multiplicity, is almost Gorenstein as a local ring. Here, however, we state more: namely, that $\calR(I)$ is almost Gorenstein as a graded ring.
\end{rem}

The first main result of this paper is stated as follows. 

\begin{thm}\label{3.3}
Let $(A, \m)$ be a two-dimensional regular local ring with infinite residue class field. Let $I$ be an $\m$-primary ideal of $A$ and assume that $I$ is not a parameter ideal. Then the following conditions are equivalent. 
\begin{itemize}
\item[$(1)$]  The Rees algebra $\calR(I)$ is nearly Gorenstein.  
\item[$(2)$] There exists a regular system of parameters $x, y$ of $A$ and an integer $n \ge 1$ such that $I=(x^2, xy, y^{n+1})$. 
\end{itemize}
When this is the case, $\calR(I)$ is an almost Gorenstein graded normal domain with minimal multiplicity. 
\end{thm}

\begin{proof}
We write $\m = (x, y)$. 
By Lemma \ref{3.1}, it suffices to show that, when $\calR$ is nearly Gorenstein, the ideal $I$ is integrally closed. To show this, we choose a minimal reduction $Q$ of $I$. Note that $I^2=QI$ and $I \ne Q$. By Proposition \ref{3.5}, we have $I=Q:_A\m$ and $I=\m(I:_A\m)$. Thus $\mu_A(I) = 3$. We write $I:_A\m = (f_1, f_2, \ldots, f_{\ell})$ with $\ell \ge 1$ and $f_i \in A$. Without loss of generality, we may assume that \(I\) is of either of the following forms:
\begin{itemize}
\item[$(\rm i)$] $I = (xf_1, xf_2, yf_1)$, or
\item[$(\rm ii)$] $I = (xf_1, xf_2, yf_3)$. 
\end{itemize}
Assume that $I = (xf_1, xf_2, yf_1)$. We proceed to show that $\overline{I} = I$. Indeed, since $I=(x, y)(f_1, f_2)$, we can write $yf_2 = a_1(xf_1) + a_2(xf_2) + a_3(yf_1)$ for some $a_1, a_2, a_3 \in A$. Then 
$$
y(f_2 - a_3f_1) = x(a_1f_1 + a_2f_2).
$$
Since $x, y$ is a regular sequence on $A$, it follows that $f_2 -a_3f_1 = cx$ with $c \in A$. Thus 
\begin{center}
$I=(x, y)(f_1, cx) = (xf_1, cx^2, yf_1, ycx)$ \ and \ $I=(xf_1, xf_2, yf_1) = (xf_1, cx^2, yf_1)$. 
\end{center}
This shows $ycx = b_1(xf_1)+b_2(cx^2)+b_3(yf_1)$ with $b_1, b_2, b_3 \in A$. Hence
$$
x(cy-b_2cx) = f_1(b_1x+b_3y). 
$$
Note that $x, f_1$ forms a regular sequence on $A$. This implies that $cy-b_2cx = df_1$ for some $d \in A$. 
We now assume that $c \in \m$. We then have $d \in \m$, so that $d = px + qy$ with $p, q \in A$. Therefore, since
$$
y(c-qf_1) = x(b_2c + pf_1)
$$
and the sequence $x, y$ is $A$-regular, $c - qf_1 = ex$ for some $e \in A$. Hence we get $(f_1, f_2) = (f_1, cx) = (f_1, ex^2)$. 
Next, taking $c=ex \in \m$ and repeating the same argument, we obtain that, for every $n \ge 2$, we can choose $e_n \in A$ such that
$$
(f_1, f_2) = (f_1, e_nx^n).
$$
This yields that $(f_1, f_2) \subseteq (f_1) + \bigcap_{n=1}^{\infty}(x^n) =(f_1)$. This makes a contradiction. Hence $c$ is a unit in $A$, so $(f_1, f_2) = (f_1, x)$. 
Since $A/xA$ is a discrete valuation ring with maximal ideal of the form $\m/xA=(\overline{y})$, we see that $(f_1, x)=(y^n, x)$ for some $n \ge 1$. Consequently, we have $I=(x, y)(y^n, x) = (x^2, xy, y^{n+1})$, and hence $I$ is integrally closed. 

Next, assume that $\overline{I} \ne I$. Then, by the argument above, we may write
$$
I=(xf_1,xf_2,yf_3).
$$
To complete the proof, we shall show that this leads to a contradiction. 

\begin{claim}\label{claim1}
For every $n \ge 0$, there exist $f_1^{(n)}, f_2^{(n)} \in A$ such that the equality
$$
I=(x, y)(f_1^{(n)}x^n, f_2^{(n)}x^n, f_3) = (f_1^{(n)}x^{n+1}, f_2^{(n)}x^{n+1}, yf_3)
$$
holds.
\end{claim}

\begin{proof}[Proof of Claim \ref{claim1}]
We prove the assertion by induction on $n \ge 0$. When $n=0$, it suffices to take $f_1^{(0)}=f_1$ and $f_2^{(0)}=f_2$. 
For the inductive step, assume that the assertion holds for $n=k~ (k\ge 0)$, and prove that it holds for $n=k+1$. Namely, we assume that $I=(x, y)(f_1^{(k)}x^k, f_2^{(k)}x^k, f_3)$ and
\begin{eqnarray*}
y f_1^{(k)}x^k \!\!&=&\!\! a_1^{(k)}f_1^{(k)}x^{k+1} + a_2^{(k)}f_2^{(k)}x^{k+1} + a_3^{(k)}yf_3 \\
y f_2^{(k)}x^k \!\!&=&\!\! b_1^{(k)}f_1^{(k)}x^{k+1} + b_2^{(k)}f_2^{(k)}x^{k+1} + b_3^{(k)}yf_3 \\
xf_3 \!\!&=&\!\! c_1^{(k)}f_1^{(k)}x^{k+1} + c_2^{(k)}f_2^{(k)}x^{k+1} + c_3^{(k)}yf_3 
\end{eqnarray*}
with $f_i^{(k)}, a_j^{(k)}, b_j^{(k)}, c_j^{(k)} \in A$. If $a_3^{(k)}$ is a unit in $A$, the first equality above shows $yf_3 \in (f_1^{(k)}x^{k+1}, f_2^{(k)}x^{k+1}, y f_1^{(k)}x^k)$. The induction hypothesis guarantees that
$$
I=(f_1^{(k)}x^{k+1}, f_2^{(k)}x^{k+1}, yf_3) = (x \cdot f_1^{(k)}x^{k}, x \cdot f_2^{(k)}x^{k}, y \cdot f_1^{(k)}x^{k}).
$$ 
This falls under case $(i)$, and hence $I$ is integrally closed, contradicting our assumption that $\overline{I} \ne I$. Therefore, $a_3^{(k)} \in \m$. Similarly, the second equality above yields $b_3^{(k)} \in \m$.
By the first equality above, we obtain $y(f_1^{(k)}x^k - a_3^{(k)}f_3) = x^{k+1}(a_1^{(k)}f_1^{(k)}+a_2^{(k)}f_2^{(k)})$, so that
$$
f_1^{(k)} x^k=a_3^{(k)}f_3 + f_1^{(k+1)}x^{k+1} \quad \quad (*)
$$
for some $f_1^{(k+1)} \in A$, because $x^{k+1}, y$ forms a regular sequence on $A$. Similarly, the second equality above induces the equality 
$$
f_2^{(k)} x^k=b_3^{(k)}f_3 + f_2^{(k+1)}x^{k+1} \quad \quad (**)
$$
holds for some $f_2^{(k+1)} \in A$. Therefore, we get
$$
I=(x, y)(f_1^{(k)}x^k, f_2^{(k)}x^k, f_3) = (x, y)(f_1^{(k+1)}x^{k+1}, f_2^{(k+1)}x^{k+1}, f_3).
$$
Substituting $(*)$ and $(**)$ into the third equality above, we obtain 
$$
xf_3 = c_1^{(k)}x (a_3^{(k)}f_3 + f_1^{(k+1)}x^{k+1}) + c_2^{(k)}x (b_3^{(k)}f_3 + f_2^{(k+1)}x^{k+1}) + c_3^{(k)}yf_3
$$
which shows 
$$
(1 - a_3^{(k)}c_1^{(k)} - b_3^{(k)}c_2^{(k)})(xf_3) = c_1^{(k)}f_1^{(k+1)}x^{k+2} + c_2^{(k)}f_2^{(k+1)}x^{k+2} + c_3^{(k)}yf_3.
$$
Since $a_3^{(k)}, b_3^{(k)} \in \m$, it follows that $xf_3 \in (f_1^{(k+1)}x^{k+2}, f_2^{(k+1)}x^{k+2}, yf_3)$. We next substitute $(*)$ and $(**)$ into the first and the second equalities respectively above to obtain 
$$
yf_1^{(k)}x^k, yf_2^{(k)}x^k \in (xf_3, f_1^{(k+1)}x^{k+2}, f_2^{(k+1)}x^{k+2},yf_3) =(f_1^{(k+1)}x^{k+2}, f_2^{(k+1)}x^{k+2},yf_3)
$$
which completes the proof. 
\end{proof}

Therefore, for each $n \ge 0$, we have $I=(f_1^{(n)}x^{n+1}, f_2^{(n)}x^{n+1}, yf_3)$ for some $f_1^{(n)}, f_2^{(n)} \in A$. Hence
$$
I \subseteq \bigcap_{n=1}^{\infty}(x^{n+1}, y) \subseteq (y)
$$
which makes a contradiction. 
\end{proof}

\begin{rem}
Let $(A,\m)$ be a two-dimensional regular local ring and $I ~(\ne A)$ an ideal of $A$ with $\height_A I>0$. Then either $I\cong A$ as an $A$-module, or $I\cong J$ for some $\m$-primary ideal $J$ of $A$.
\end{rem}

\begin{ex}\label{3.5a}
Let $A=k[[X, Y]]$ be the formal power series ring over a field $k$. 
We set $I = (X^2, XY^2, Y^3)$ and $Q=(X^2, Y^3)$. Then $Q$ is a parameter ideal of $A$ and $I^2=QI$. Then $I=Q:_A\m$, but $I \not\subseteq\m (I:_A\m) = \m^3$. Thus, $\calR(I)$ is not nearly Gorenstein. 
\end{ex}

\begin{ex}\label{3.7}
Let $A=k[[X, Y]]$ be the formal power series ring over a field $k$, and let $\ell \ge 1$ be an integer. Set $I=\m^{\ell}$, $Q=(X^{\ell}, Y^{\ell})$, and $J=Q:_AI$. Then $I^2=QI$ and $J=\m^{\ell-1}=I:_A\m$. Hence, $\calR(I)$ is nearly Gorenstein if and only if $\ell \le 2$. 
On the other hand, \cite[Corollary 1.4]{GMTY2} shows that, provided $k$ is infinite, the Rees algebra $\calR(I)$ is almost Gorenstein for every $\ell \ge 1$. 
\end{ex}

\begin{ex}
Let $A=k[[X, Y]]$ be the formal power series ring over a field $k$. We set $Q=(X^3, Y^3)$ and $I=Q:_A\m$. Then $I=(X^3, X^2Y^2, Y^3)$, $I^2=QI$, and $\overline{I} = \m^3 \ne I$. Thus $\calR(I)$ is not nearly Gorenstein. 
\end{ex}

We consider the case where $A$ is not regular. 
Recall that a Noetherian local ring $(A, \m)$ is called a {\it hypersurface} if $\mu_A(\m) - \depth A \le 1$ (\cite[page 44]{Avramov}). 

\begin{thm}\label{hyp}
Let $(A, \m)$ be a two-dimensional non-regular Gorenstein local ring with infinite residue class field. Let $I$ be an integrally closed $\m$-primary ideal of $A$. 
Then the following conditions are equivalent. 
\begin{itemize}
\item[$(1)$] The Rees algebra $\calR(I)$ is a non-Gorenstein nearly Gorenstein ring.  
\item[$(2)$] $I=\m$ and $A$ is a hypersurface  with $\rme(A) = 2$. 
\end{itemize}
\end{thm}

\begin{proof}
$(2) \Rightarrow (1)$ Since $A$ is a two-dimensional hypersurface with $\rme(A)=2$, it has minimal multiplicity. As the residue field $A/\m$ is infinite, there exists a parameter ideal $Q$ of $A$ such that $\m^2 =Q\m$. In particular, $\m=Q:_A\m$, and hence $I=Q:_A\m$. Moreover, since $I=\m$, we have $I:_A\m=A$, and therefore $I=\m(I:_A\m)$. Consequently, Proposition \ref{3.5} applies and shows that the Rees algebra $\calR$ is nearly Gorenstein. 
Finally, $\calR$ is not Gorenstein because $I=\m\neq Q$; see Remark \ref{2.8}.

$(1) \Rightarrow (2)$ Let $Q$ be a parameter ideal of $A$ which is a reduction of $I$. Since $\calR$ is not Gorenstein, we have $I\neq Q$, whence $\overline{Q} \ne Q$. By Proposition \ref{3.5}, the nearly Gorenstein property of $\calR(I)$ yields that $I = Q:_A\m$ and $I=\m(I:_A \m)$. Assume that $A$ is not a hypersurface. Then $\mu_A(\m)\ge 4$. Since $I$ is integrally closed and $\sqrt{I}=\m$, the ideal $I$ is $\m$-full (\cite[Theorem (2.4)]{G}). Hence we get
$$
\mu_A(I)\ge \mu_A(\m)\ge 4
$$
by \cite[Lemma (2.2) (2)]{G} (see also \cite[Theorem 3]{Watanabe}). 
On the other hand, since $I=Q:_A\m$ and $A/Q$ is Gorenstein, we obtain $\ell_A(I/Q)=1$, and hence $\mu_A(I)=\mu_A(Q)+1=3$, a contradiction. 
Therefore $A$ is a hypersurface.
By \cite[(40.6) Theorem]{N}, we have $\rme(A)\ge 2$. 
Since $A$ is a two-dimensional hypersurface, we may write
$$
G(\m)\cong k[X,Y,Z]/(F),
$$
where $k=A/\m$ and $F\in k[X,Y,Z]$ is a homogeneous polynomial.
Moreover, $\deg F=\rme(A)\ge 2$; see \cite[Exercise 4.6.12(a)]{BH}. 
Consequently
$$
\mu_A(\m^2)=\dim_k G(\m)_2= \dim_k k[X,Y,Z]_2-\dim_k (F)_2\ge 6-1=5.
$$
If $I\neq \m$, then the equality $I=\m(I:_A\m)$ implies $I\subseteq \m^2$. 
Therefore
$$
3=\mu_A(I)\ge \mu_A(\m^2)\ge 5
$$
which is impossible. Hence $I=\m$. 
This implies $\m=Q:_A\m$. 
Since $A$ is not regular, it follows that $\m^2=Q\m$ by \cite[Theorem 2.2]{CP}. 
Thus $A$ has minimal multiplicity, and hence $\rmr(A)=\rme(A)-1$ (\cite[3.1 Proposition]{Sally}). Since $A$ is Gorenstein, we conclude that $\rme(A)=2$, as desired.
\end{proof}

We conclude this section with the following remark concerning a case where the ideal is not necessarily of reduction number one.

\begin{rem}
Let $(A, \m)$ be a regular local ring with $d=\dim A \ge 3$. Then the Rees algebra $\calR(\m^{d-2})$ is nearly Gorenstein, as will be shown in a forthcoming paper (\cite{EY2}). 
Note that, assuming $A/\m$ is infinite, the Rees algebra $\calR(\m^{\ell})$ is Gorenstein if and only if it is an almost Gorenstein graded ring, and this happens precisely when $\ell=d-1$ (\cite[Theorem 4.1(2)]{Ooishi}, \cite[Theorem 1.6]{GMTY4}).
\end{rem}

%%%%%%%%%%%%%%%%%%%%%%%%%%%%%%%%%%%%%%%%%%%%%%%%%%%%%%%%%%%%%%%%%%%%%%%%%%%%%%%%%%%%%%%%%%%%%%%%%%%%%%%%%%%%%%%%%%%%%%%%%%%%%%%%%%%%%%%%%%%%%%%%

\section{Nearly Gorenstein extended Rees algebras}

In this section, we focus on the nearly Gorenstein property of extended Rees algebras.

\begin{setup}
Let $(A, \m)$ be a Gorenstein local ring with $d=\dim A \ge 2$ and $K=\rmQ(A)$ its total ring of fractions. Let $I$ be an $\m$-primary ideal of $A$ and assume that $I$ contains a parameter ideal $Q$ of $A$ such that $I^2=QI$. We set $J=Q:_A I$. Let  
$$
\calR'=\calR'(I) =A[It, t^{-1}] \subseteq A[t, t^{-1}] \ \ \text{and} \ \ G=G(I) = \calR'(I)/t^{-1}\calR'(I) 
$$
be the extended Rees algebra and the associated graded ring of $I$, respectively, where $t$ is an indeterminate over $A$.
Denote by $\fkM = (t^{-1}, \m, It)\calR'$ the graded maximal ideal of $\calR'$.
\end{setup}

Since $I^2=QI$, it follows from \cite[Corollary 2.7]{VV} that $G$ is Cohen-Macaulay; consequently, $\calR'$ is Cohen-Macaulay as well. Moreover, the graded canonical module of $\calR'$ is given by
$\rmK_{\calR'} \cong (t^{d-2}, Jt^{d-1})\calR'$ (\cite[Theorem 2.7 (b)]{U}). 
Therefore, after shifting degrees, we have $\rmK_{\calR'}(d-1)\cong X$, 
where $X=(t^{-1},\, J)\calR'$. Let $F=\rmQ(\calR')$ be the total ring of fractions of $\calR'$, and set $L=\calR':_F X$. Then
$$
L = (\calR':_F t^{-1}\calR') \cap (\calR':_F J\calR') = t\calR' \cap (\calR':_F J\calR') = t\calR' \cap (\calR':_{T^{-1}\calR'} J\calR')
$$
where $T$ denotes the set of homogeneous non-zerodivisors on $\calR'$. 
In particular, $L$ is a graded $\calR'$-submodule of $T^{-1}\calR' = K[t, t^{-1}]$.

\begin{lem}
For each $n \in \Bbb Z$,  the equality $L_n = (I^{n-1} \cap (I^n:_AJ))t^n$ holds in $A[t, t^{-1}]$. 
\end{lem}

\begin{proof}
Let $x \in K[t, t^{-1}]$ be a homogeneous element of degree $n$. Write $x = ct^n$ with $c \in K$. Then 
$x \in (\calR':_F J\calR')_n$ if and only if $(ct^n)(JI^{\ell}t^{\ell}) \subseteq I^{n+\ell}t^{n+\ell}$ for all $\ell \in \Bbb Z$. The latter is equivalent to $c  J \subseteq I^n$, or equivalently, $c \in I^n:_K J = I^n:_A J$. Consequently,
$(\calR':_F J\calR')_n = (I^n:_A J)t^n$. 
Taking degree $n$ components in the equality $L=t\calR'\cap (\calR':_F J\calR')$, we obtain
$$
L_n = (t\calR')_n \cap (\calR':_F J\calR')_n= I^{n-1}t^n \cap (I^n:_A J)t^n= (I^{n-1}\cap (I^n:_A J))t^n
$$
as required.
\end{proof}

\begin{prop}\label{4.3}
The following conditions are equivalent. 
\begin{itemize}
\item[$(1)$] The extended Rees algebra $\calR'(I)$ is nearly Gorenstein.
\item[$(2)$]  The associated graded ring $G(I)$ is nearly Gorenstein.  
\item[$(3)$]  $\m \subseteq (I:_AJ) + J$ and $I = (I^2:_AJ) + J(I:_AJ)$. 
\end{itemize}
\end{prop}

\begin{proof}
$(1) \Leftrightarrow (2)$ Let $\rmK_G$ be the graded canonical module of $G$. Note that $t^{-1}\in X\cdot L=\tr_{\calR'}(\rmK_{\calR'})$. 
Localizing at the graded maximal ideal, we may apply \cite[Proposition 3.8]{Miyazaki} to the homogeneous non-zerodivisor $t^{-1}$ on $\calR'$. Since $G\cong \calR'/(t^{-1})$, it follows that $\tr_G(\rmK_G) = \tr_{\calR'}(\rmK_{\calR'})G$. Therefore, the nearly Gorenstein property for $\calR'$ is equivalent to that for $G$.

$(1) \Leftrightarrow (3)$ Since $\tr_{\calR'}(\rmK_{\calR'})=X\cdot L=t^{-1}L + JL$, we have, for each $n\in \Bbb Z$, 
$$
[\tr_{\calR'}(\rmK_{\calR'})]_n = [t^{-1}L + JL]_n=t^{-1}L_{n+1} + J L_n. 
$$ 
Assume that $\calR'$ is nearly Gorenstein. Then $\fkM\subseteq \tr_{\calR'}(\rmK_{\calR'})$, and comparing the degree $0$ part we obtain
$$
\fkm = \fkM_0 \subseteq [\tr_{\calR'}(\rmK_{\calR'})]_0 = t^{-1}L_{1} + J L_0 = (I:_AJ) + J
$$
where the last equality follows from $L_0=A:_A J=A$. Similarly, comparing the degree $1$ part yields
$$
It = \fkM_1 \subseteq [\tr_{\calR'}(\rmK_{\calR'})]_1 = t^{-1}L_{2} + J L_1 = (I \cap (I^2:_AJ))t + J(I:_AJ)t.
$$
As $A/Q$ is an Artinian Gorenstein ring and $I^2=QI$, it follows that $I^2:_AJ \subseteq  I$, and hence $I = (I^2:_AJ) + J(I:_AJ)$. 
Conversely, assume $(3)$. Then the conditions in $(3)$ imply $\m\subseteq \tr_{\calR'}(\rmK_{\calR'})$ and $It\subseteq \tr_{\calR'}(\rmK_{\calR'})$. 
Moreover, since $\fkM=(\m,It,t^{-1})\calR'$ and $t^{-1}\in \tr_{\calR'}(\rmK_{\calR'})$, it follows that $\fkM\subseteq \tr_{\calR'}(\rmK_{\calR'})$. 
Hence $\calR'(I)$ is nearly Gorenstein.
\end{proof}

\begin{rem}\label{4.4}
Assume $I \ne Q$. By \cite[Corollary 2.11 (b)]{U} (see also \cite[Proposition (2.2)]{GIW}), the extended Rees algebra $\calR'(I)$ is Gorenstein if and only if $I=J$, i.e., $I$ is a good ideal in the sense of \cite[page 2310]{GIW}. Proposition \ref{4.3} shows that $\calR'(I)$ is nearly Gorenstein but not Gorenstein if and only if $\m = (I:_AJ) + J$ and $I = (I^2:_AJ) + J(I:_AJ)$.
\end{rem}

\begin{cor}\label{4.5}
If the Rees algebra $\calR(I)$ is nearly Gorenstein, then the extended Rees algebra $\calR'(I)$ is also nearly Gorenstein.
\end{cor}

\begin{proof}
By \cite[Corollary 3.7]{Ikeda}, we may assume that $\calR(I)$ is not Gorenstein. 
Recall from Proposition \ref{4.3} that $\calR'(I)$ is nearly Gorenstein if and only if $\m \subseteq (I:_A J)+J$ and $I=(I^2:_A J)+J(I:_A J)$. 
Since $\calR(I)$ is nearly Gorenstein, Proposition \ref{2.13} shows that $d\le 4$. 
We verify the two conditions case by case. If $d=2$, Proposition \ref{3.5} yields $I=Q:_A\m$ and $I=\m(I:_A\m)$. In particular, $J=Q:_A I=\m$, so the above conditions are satisfied.
Assume that $d=3$. Proposition \ref{2.13} shows that $\m=I:_A J$ and 
$I=(I^2:_A J)+\m J$. Hence the conditions in Proposition \ref{4.3} hold. 
Finally, suppose $d=4$. Then Proposition \ref{2.13} gives $I=\m$. 
If $J=A$ the conditions are immediate, while if $J\ne A$ then $I=J=\m$ and they again follow. Therefore $\calR'$ is nearly Gorenstein.
\end{proof}

\begin{ex}\label{4.6}
Let $A=k[[X, Y]]$ be the formal power series ring over a field $k$, and let $\ell \ge 1$ be an integer. Set $I=\m^{\ell}$, $Q=(X^{\ell}, Y^{\ell})$, and $J=Q:_AI$. Then $I^2=QI$, $J=\m^{\ell-1}$, and $I:_AJ = \m$. 
Hence, $\calR'(I)$ is nearly Gorenstein for every $\ell \ge 1$. Therefore, in view of Example \ref{3.7}, the converse of Corollary \ref{4.5} does not hold.
\end{ex}

\begin{ex}\label{4.7}
Let $A=k[[X, Y]]$ be the formal power series ring over a field $k$. Set $I = (X^2, XY^2, Y^3)$ and $Q=(X^2, Y^3)$. By Example \ref{3.5a}, we have  $I^2=QI$,  $I=Q:_A\m$, and $I:_A\m=\m^2$.  Consequently, $J=Q:_AI =\m$ and  $I:_AJ  = \m^2$. On the other hand, a direct computation shows that $I^2:_A\m = (X^4, X^3Y, X^2Y^2, XY^4, Y^5)$. Therefore
$$
(I^2:_A J)+J(I:_A J)
=\m^3
$$
and hence $I\not\subseteq (I^2:_A J)+J(I:_A J)$. 
By Proposition \ref{4.3}, it follows that $\calR'(I)$ is not nearly Gorenstein.
\end{ex}

We now turn to the nearly Gorenstein property of extended Rees algebras over two-dimensional regular local rings. 

Unless otherwise specified, let $(A, \m)$ be a two-dimensional regular local ring with infinite residue class field. Let $I$ be an integrally closed $\m$-primary ideal of $A$, and assume that $I$ contains a minimal reduction $Q$ such that $\overline{Q} \ne Q$. Set $J=Q:_AI$. %Recall that, for each ideal $I$ of $A$, we define $\rmo(I) = \max\{n \in \Bbb Z \mid I \subseteq \m^n\}$.
By \cite[Corollary (6.15)]{GIW}, it should be noted that no good ideals exist in this setting. Hence the extended Rees algebra is never Gorenstein; see Remark \ref{4.4}. In particular, $I \ne J$.

We begin with the case in which the order of the ideal $I$ is at least $3$. Here, we say that an ideal is {\it simple} if it cannot be expressed as the product of two proper ideals. 

\begin{prop}\label{4.8}
Assume that $\rmo(I) \ge 3$. Then the following conditions are equivalent. 
\begin{itemize}
\item[$(1)$] The extended Rees algebra $\calR'(I)$ is nearly Gorenstein. 
\item[$(2)$] $I=\m^{\ell}$ for some $\ell \ge 3$.
\item[$(3)$] $I=\m J$. 
\item[$(4)$] $\m = I:_AJ$. 
\item[$(5)$] $J = I:_A\m$. 
\end{itemize}
\end{prop}

\begin{proof}
Since $I$ is integrally closed, it is contracted. Hence $\mu_A(I)=\rmo(I)+1\ge 4$. 
Set $n=\mu_A(I)$. Since $\height_A I=2$, the Hilbert-Burch theorem yields an exact sequence
$$
0 \to A^{\oplus (n-1)} \overset{\Bbb M}{\longrightarrow} A^{\oplus n} \to A \to A/I \to 0
$$
of $A$-modules such that $I=\rmI_{n-1}(\Bbb M)$, where $\Bbb M$ is an $n\times (n-1)$-matrix with entries in $\m$. Then \cite[(3.1) Lemma, (3.3) Proposition]{Huneke-Swanson} shows that
$$
J=Q:_AI = \rmI_{n-2}(\Bbb M)
$$
and that $J$ is integrally closed. Since $n\ge 4$, it follows that
$J\subseteq \m^{n-2}\subseteq \m^2$.

$(1)\Rightarrow(4)$ Since $\calR'(I)$ is nearly Gorenstein, one has $\m=(I:_AJ)+J$. Since $J\subseteq \m^2$, it follows that $\m=(I:_AJ)+\m^2$, whence $\m=I:_AJ$. 

$(5)\Rightarrow(4)$ This follows from the fact that $I\ne \m$ and $I\ne J$.

$(4)\Rightarrow(3)$ Assume that $\m=I:_AJ$. Then $\m J\subseteq I$.  Moreover, we have
\[
I=\rmI_{n-1}(\Bbb M)\subseteq \m\cdot \rmI_{n-2}(\Bbb M)=\m J
\]
which shows $I=\m J$. %If $I=\m J$, then one has $I:_AJ=\m$ because $I\ne J$. 

$(3)\Rightarrow(2)$ Since $I$ is integrally closed, Zariski's factorization theorem (\cite[Appendix 5, Theorem 3]{ZS}, \cite[Theorem 3.9]{Huneke}, \cite[Theorem 14.4.9]{SH}) yields a factorization
\[
I=I_1I_2\cdots I_\ell,
\]
where $\ell\ge1$ and each $I_i$ is a simple integrally closed $\m$-primary ideal.
We first show that $I$ is not simple. Indeed, if $I$ were simple, then from $I=\m J$ it would follow that $J=A$, and hence $I=\m$, contradicting the assumption $\rmo(I)\ge3$. Thus $I$ is not simple.
We next show that $I_i=\m$ for every $1\le i\le \ell$. Fix $i$. By reordering the factors if necessary, we may assume that $i=1$. Set $I'=I_2\cdots I_\ell$. Then $I'$ is again an integrally closed $\m$-primary ideal. By \cite[Proposition (3.3)]{Lipman} (see also \cite[Proposition 18.6.4]{SH}), one has
$$
J=Q:_AI=\adj(I)
$$
where $\adj(I)$ denotes the adjoint ideal of $I$. Hence
$$
J:_AI'=\adj(I):_AI'=\adj(I_1I'):_AI'=\adj(I_1),
$$
where the last equality follows from \cite[Proposition 18.2.1]{SH}.
On the other hand, since $I=\m J$ and $J$ is integrally closed, we have
$I:_A\m=\m J:_A\m=J$ (see \cite[Corollary 6.8.7]{SH}). Therefore, using the fact that $I_1$ is integrally closed, we obtain
\begin{eqnarray*}
\adj(I_1) \!\!&=&\!\! J:_AI' =(I:_A\m):_AI' = I:_A \m I' = (I:_AI'):_A\m \\
\!\!&=&\!\! (I_1 I':_AI'):_A\m = I_1:_A \m.
\end{eqnarray*}
Thus, by the implication $(5)\Rightarrow(3)$ already proved, we get
$I_1 = \m\adj(I_1)$. Since $I_1$ is simple, it follows that $\adj(I_1)=A$, and therefore $I_1=\m$. 
As $i$ was arbitrary, we conclude that $I_i=\m$ for all $1\le i\le \ell$. Consequently
$$
I=I_1I_2\cdots I_\ell=\m^\ell.
$$
Since $\rmo(I)\ge3$, necessarily $\ell\ge3$.

$(2)\Rightarrow(1), (5)$ Let $x,y$ be a regular system of parameters of $A$, and set
$I=\m^\ell$, $Q=(x^\ell,y^\ell)$. Then the same argument as in Example \ref{4.6} shows that the extended Rees algebra $\calR'(I)$ is nearly Gorenstein. The direct computation shows $J=\m^{\ell-1} = I:_A\m$. 
\end{proof}

We next consider the case where $\rmo(I) = 2$ and $I$ is not simple. 

\begin{lem}\label{4.9}
Assume that $\rmo(I) = 2$. If the extended Rees algebra $\calR'(I)$ is nearly Gorenstein, then $J=\m$.
\end{lem}

\begin{proof}
Since $\rmo(I)=2$ and $\overline{I}=I$, we obtain $\mu_A(I)=3$. 
By the proof of Proposition \ref{4.8}, it follows that $I=\rmI_2(\Bbb M)$ and $J=\rmI_1(\Bbb M)$. Thus $I \subseteq J^2$. Hence $I:_AJ \subseteq J^2:_AJ = J$ (\cite[Corollary 6.8.7]{SH}). The nearly Gorenstein property of $\calR'(I)$ implies $\m \subseteq (I:_AJ) + J = J \subsetneq A$. Therefore $J=\m$.
\end{proof}

\begin{prop}\label{4.10}
Assume that $\rmo(I) = 2$ and $I$ is not simple. Then the following conditions are equivalent. 
\begin{itemize}
\item[$(1)$] The extended Rees algebra $\calR'(I)$ is nearly Gorenstein. 
\item[$(2)$] There exists a regular system of parameters $x, y$ of $A$ and an integer $n \ge 1$ such that $I=(x^2, xy, y^{n+1})$. 
\end{itemize}
\end{prop}

\begin{proof}
$(2) \Rightarrow (1)$ This follows from Theorem \ref{3.3} and Corollary \ref{4.5}.

$(1) \Rightarrow (2)$ By Lemma \ref{4.9}, we have $J=\m$. Hence $I=(I^2:_A\m) + \m (I:_A\m)$. Since $\rmo(I)=2$, one has $I\subseteq \m^2$, and therefore $I^2:_A\m \subseteq \m^4:_A\m=\m^3$, i.e., $\rmo(I^2:_A\m) \ge 3$. Suppose that $I:_A\m\subseteq \m^2$. Then $\m(I:_A\m)\subseteq \m^3$, so that $I=(I^2:_A\m)+\m(I:_A\m)\subseteq \m^3$, 
which contradicts $\rmo(I)=2$. Thus $I:_A\m\nsubseteq \m^2$. Since $I\ne \m$, it follows that $\rmo(I:_A\m)=1$. 
Since $\rmo(I)=2$ and $I$ is not simple, by Zariski's factorization theorem (\cite[Appendix 5, Theorem 3]{ZS}, \cite[Theorem 3.9]{Huneke}, \cite[Theorem 14.4.9]{SH}), we may write
$$
I=I_1 I_2
$$
where each $I_i$ is a simple $\m$-primary integrally closed ideal with $\rmo(I_i)=1$. We now identify $G(\m)=\bigoplus_{n\ge0}\m^n/\m^{n+1}\cong k[X,Y]$, where $k[X,Y]$ denotes the polynomial ring over the field $k=A/\m$. For each $i=1,2$, consider the $k$-subspace
$$
W_i=(I_i+\m^2)/\m^2 \subseteq G(\m)_1=\m/\m^2.
$$
Since $\rmo(I_i)=1$, we have $1\le \dim_kW_i\le 2$.

\begin{claim}\label{claim2}
Either $I_1=\m$ or $I_2=\m$.
\end{claim}

\begin{proof}[Proof of Claim \ref{claim2}]
Suppose, to the contrary, that $I_1\ne \m$ and $I_2\ne \m$. Since $G(\m)_1=\m/\m^2$ is two-dimensional over $k$, Nakayama's lemma yields
$\dim_kW_i=1$ for each $i=1, 2$. 
Choose $f_i\in I_i\setminus \m^2$ such that $W_i=k\cdot \overline{f_i}$. 
We then have
$$
[I_1I_2+\m^3]/\m^3=W_1W_2=k\cdot \overline{f_1f_2}
$$
in $G(\m)\cong k[X,Y]$, and hence $\dim_k([I+\m^3]/\m^3)=1$. 
On the other hand, since $I=(I^2:_A\m)+\m(I:_A\m)$ and $I^2:_A\m\subseteq \m^3$, we get $[I+\m^3]/\m^3=[\m(I:_A\m)+\m^3]/\m^3$.
Set
$$
V=[(I:_A\m)+\m^2]/\m^2\subseteq G(\m)_1.
$$
As $\rmo(I:_A\m)=1$, we get $1\le \dim_kV\le 2$. Note that 
$[\m(I:_A\m)+\m^3]/\m^3=G(\m)_1\cdot V\subseteq G(\m)_2$. If $\dim_kV=1$, then $V=k\cdot \overline{\xi}$ for some $\xi\in (I:_A\m)\setminus \m^2$, and therefore
$$
G(\m)_1\cdot V
=
k\cdot X\overline{\xi}+k\cdot Y\overline{\xi}
$$
which is two-dimensional over $k$. If\/ $\dim_kV=2$, then $V=G(\m)_1$, so $G(\m)_1\cdot V=G(\m)_2$. In either case, $\dim_k\bigl([\m(I:_A\m)+\m^3]/\m^3\bigr)=\dim_k(G(\m)_1\cdot V)\ge 2$. 
This contradicts the fact that $\dim_k([I+\m^3]/\m^3)=1$. Hence, either $I_1=\m$ or $I_2=\m$.
\end{proof}

By Claim \ref{claim2}, we may assume without loss of generality that $I_1=\m$. Since $\rmo(I_2)=1$, there exists a regular system of parameters $x, y$ of $A$ and an integer $n\ge 1$ such that $I_2=(x,y^n)$. Therefore, $I=I_1I_2=\m(x,y^n)=(x^2,xy,y^{n+1})$, as required.
\end{proof}

Finally, we turn to the case where $\rmo(I) = 2$ and $I$ is simple.

\begin{prop}\label{4.11}
Assume that $\rmo(I) = 2$ and $I$ is simple. Then the extended Rees algebra $\calR'(I)$ is not nearly Gorenstein. 
\end{prop}

\begin{proof}
By Lemma \ref{4.9}, we note that $J=\m$. Hence, if $\calR'$ were nearly Gorenstein, then $I=(I^2:_A\m)+\m(I:_A\m)$. As in the proof of Proposition \ref{4.10}, one has $[I+\m^3]/\m^3=[\m(I:_A\m)+\m^3]/\m^3$ and  $\rmo(I:_A\m)=1$. Choose $f\in (I:_A\m)\setminus \m^2$. Then
$f\m\subseteq \m(I:_A\m)$, so, identifying $G(\m)=\bigoplus_{n\ge0}\m^n/\m^{n+1}\cong k[X,Y]$ as the polynomial ring over $k=A/\m$, we obtain
$$
(f \m +\m^3)/\m^3 \subseteq [\m (I:_A\m)+\m^3]/\m^3 = [I+\m^3]/\m^3. 
$$
Since $(f\m+\m^3)/\m^3$ is generated, as a $k$-vector space, by $X\overline{f}$ and $Y\overline{f}$, it follows that
$$
X\overline{f},\,Y\overline{f}\in [I+\m^3]/\m^3.
$$
Let $c(I)$ denote the characteristic form of $I$, that is, the greatest common divisor of the elements of $[I+\m^3]/\m^3$ (see \cite[Remark 14.1.1 (6)]{SH}). Then $\overline{f}\mid c(I)$, and we can write $c(I)=\overline{f}\cdot\xi$
for some $\xi\in G(\m)$. Since $c(I)$ is a homogeneous polynomial of degree at most $\rmo(I)=2$ and $\deg \overline{f}=1$, the element $\xi$ is also homogeneous. 
If, moreover, we assume $\overline{f}^{2}\mid c(I)$. Since $c(I)$ divides every element of $[I+\m^3]/\m^3$, in particular $c(I)\mid X\overline{f}$ and $c(I)\mid Y\overline{f}$, it follows that $\overline{f}\mid X$ and $\overline{f}\mid Y$. This is impossible, because \(\deg \overline{f}=1\) and \(X,Y\) are relatively prime in $G(\m)\cong k[X,Y]$. Thus, the square $\overline{f}^2$ does not divide $c(I)$. 
Note that $c(I)$ is a power of an irreducible homogeneous polynomial (\cite[Exercise 14.6]{SH}). Therefore $\deg c(I)=1$. However, since $I$ is simple and integrally closed, \cite[Proposition 2.5]{Huneke} (see also \cite[Proposition 14.1.12]{SH}) yields $\deg c(I)=\rmo(I)=2$. 
This contradiction shows that \(\calR'(I)\) cannot be nearly Gorenstein.
\end{proof}

Combining the above observations, we reach the following which is the second main result of this paper. 

\begin{thm}\label{4.12}
Let $(A, \m)$ be a two-dimensional regular local ring with infinite residue class field. Let $I$ be an integrally closed $\m$-primary ideal of $A$ and assume that $I$ is not a parameter ideal. Then the following conditions are equivalent. 
\begin{itemize}
\item[$(1)$] The extended Rees algebra $\calR'(I)$ is nearly Gorenstein.  
\item[$(2)$] The associated graded ring $G(I)$ is nearly Gorenstein.  
\item[$(3)$] Either $I=\m^{\ell}$ for some $\ell \ge 2$, or else there exists a regular system of parameters $x, y$ of A and an integer $n \ge 1$ such that $I=(x^2, xy, y^{n+1})$. 
\end{itemize}
\end{thm}

\begin{proof}
Let $Q$ be a minimal reduction of $I$. Since $I \ne Q$, it follows that $\overline{Q}\ne Q$.

$(1) \Leftrightarrow (2)$ We may assume that $G(I)$ is Cohen-Macaulay. By Lemma \ref{2.7}, we have $I^2=QI$. The equivalence then follows from Proposition \ref{4.3}.

$(1) \Leftrightarrow (3)$ Suppose $\rmo(I)=1$. Since $I$ is integrally closed, it is contracted, and hence $\mu_A(I)=\rmo(I)+1=2$. Thus $I$ is a parameter ideal, contrary to our assumption. Therefore, $\rmo(I)\ge 2$. The desired equivalence now follows from Propositions \ref{4.8}, \ref{4.10}, and \ref{4.11}.
\end{proof}

\begin{rem}\label{4.13}
When $\rmo(I) = 2$, the extended Rees algebra $\calR'(I)$ is nearly Gorenstein if and only if so is the Rees algebra $\calR(I)$. 
\end{rem}

\begin{conj}
Let $(A,\m)$ be a two-dimensional regular local ring with infinite residue field, and let $I$ be a non-parameter $\m$-primary ideal of $A$. If the extended Rees algebra $\calR'(I)$ is nearly Gorenstein, then $I$ is integrally closed.
\end{conj}

%%%%%%%%%%%%%%%%%%%%%%%%%%%%%%%%%%%%%%%%%%%%%%%%%%%%%%%%%%%%%%%%%%%%%%%%%%%%%%%%%%%%%%%%

%%%%%%%%%%%%%%%%%%%%%%%%%%%%%%%%%%%%%%%%%%%%%%%%%%%%%%%%%%%%%%

\end{document}